\documentclass[10pt]{article}
\usepackage{amsthm,amssymb,amsmath,amsfonts}
\usepackage{graphicx} 
\usepackage[margin=1.5in]{geometry}
\usepackage{enumitem}
\usepackage{graphics,enumitem,mathrsfs,mathtools,soul,csquotes}
\usepackage[color=red!40]{todonotes}
\usepackage{bm,dsfont}
\usepackage{tikz-cd}
\usetikzlibrary{positioning,arrows,scopes}

\definecolor{cadmiumgreen}{rgb}{0.0, 0.42, 0.24}
\usepackage[
colorlinks,citecolor=cadmiumgreen,
pagebackref,
pdfauthor={Aleyah Dawkins,
Vishal Gupta, Mark Kempton, William Linz, 
Jeremy Quail, Harry Richman, Zachary Stier}, 
pdfstartview ={FitV},
]{hyperref}
\hypersetup{
pdftitle={Node resistance curvature in Cartesian graph products},
}

\newtheorem*{thm*}{Theorem}

\theoremstyle{definition}

\newcommand{\lpr}[1]{\left(#1\right)}

\newcommand{\lcr}[1]{\left\{#1\right\}}

\theoremstyle{definition}
\newtheorem{theorem}[equation]{Theorem}
\newtheorem{proposition}[equation]{Proposition}

\newtheorem{lemma}[equation]{Lemma}

\newtheorem{conjecture}[equation]{Conjecture}
\newtheorem{example}[equation]{Example}
\newcommand{\appref}[1]{\hyperref[#1]{Appendix \ref{#1}}}
\newcommand{\claimref}[1]{\hyperref[#1]{Claim \ref{#1}}}
\newcommand{\conjref}[1]{\hyperref[#1]{Conjecture \ref{#1}}}
\newcommand{\corref}[1]{\hyperref[#1]{Corollary \ref{#1}}}
\newcommand{\defref}[1]{\hyperref[#1]{Definition \ref{#1}}}
\newcommand{\rmkref}[1]{\hyperref[#1]{Remark \ref{#1}}}
\newcommand{\exref}[1]{\hyperref[#1]{Example \ref{#1}}}
\newcommand{\figref}[1]{\hyperref[#1]{Figure \ref{#1}}}
\newcommand{\hwref}[1]{\hyperref[#1]{Homework \ref{#1}}}
\newcommand{\lemref}[1]{\hyperref[#1]{Lemma \ref{#1}}}
\newcommand{\probref}[1]{\hyperref[#1]{Problem \ref{#1}}}
\newcommand{\propref}[1]{\hyperref[#1]{Proposition \ref{#1}}}
\newcommand{\secref}[1]{\hyperref[#1]{\S\ref{#1}}}
\newcommand{\tabref}[1]{\hyperref[#1]{Table \ref{#1}}}
\newcommand{\thmref}[1]{\hyperref[#1]{Theorem \ref{#1}}}

\newcommand{\bp}{{\boldsymbol p}}

\newcommand{\bv}{{\boldsymbol v}}

\newcommand{\ga}{\alpha}     
\newcommand{\gd}{\delta}     
\newcommand{\gl}{\lambda}    
\newcommand{\gk}{\kappa}    
\newcommand{\gw}{\omega}      

\newcommand{\ket}[1]{\left\lvert#1\right\rangle}
\newcommand{\bra}[1]{\left\langle#1\right\rvert}
\newcommand{\braket}[2]{\left\langle#1|#2\right\rangle}

\newcommand{\half}{\frac{1}{2}}

\newcommand{\R}{\mathbb{R}}

\newcommand{\Z}{\mathbb{Z}}

\newcommand{\N}{\mathbb{N}}

\newcommand{\inv}{^{-1}}
\newcommand{\id}{\mathrm{id}}
\newcommand{\one}{{\bf1}}

\newcommand{\norm}[1]{\left\|#1\right\|}

\usepackage{listings}
\definecolor{codegreen}{rgb}{0,0.6,0}
\definecolor{codegray}{rgb}{0.5,0.5,0.5}
\definecolor{codepurple}{rgb}{0.58,0,0.82}
\definecolor{codeblue}{rgb}{0,0,0.95}
\definecolor{backcolor}{rgb}{0.95,0.95,0.92}
\usepackage{caption}
\usepackage{subcaption}

\usepackage{authblk}
\title{Node resistance curvature in Cartesian graph products}
\author{Aleyah Dawkins}
\affil{George Mason University}
\author{Vishal Gupta}
\affil{University of Delaware}
\author{Mark Kempton}
\affil{Brigham Young University}
\author{William Linz}
\affil{University of South Carolina}
\author{Jeremy Quail}
\affil{University of Vermont}
\author{Harry Richman}
\affil{Fred Hutchinson Cancer Center}
\author{Zachary Stier}
\affil{University of California, Berkeley}
\date{}

\begin{document}

\maketitle

\begin{abstract}
    Devriendt and Lambiotte recently introduced the \emph{node resistance curvature}, a notion of graph curvature based on the effective resistance matrix. In this paper, we begin the study of the behavior of the node resistance curvature under the operation of the Cartesian graph product. We study the natural question of global positivity of node resistance curvature of the Cartesian product of positively-curved graphs, and prove that, whenever $m,n\ge3$, the node resistance curvature of the interior vertices of a $m\times n$ grid is always nonpositive, while it is always nonnegative on the boundary of such grids. For completeness, we also prove a number of results on node resistance curvature in $2\times n$ grids and exhibit a counterexample to a generalization. We also give generic bounds and suggest several further questions for future study. 
\end{abstract}


\section{Introduction}

There has been much interest in studying analogues of properties of Riemannian manifolds in the context of finite graphs (see the book of Chung~\cite{Chung}). For example, there are a number of different notions of Ricci curvature of graphs that have been studied~\cite{ollivier, LLY, forman}. Quite recently, Devriendt and Lambiotte~\cite{DL} defined and studied a new notion of Ricci curvature of graphs based on the effective resistance matrix (Devriendt, Ottolini and Steinerberger~\cite{DOS} also defined a closely related notion of curvature). We introduce their definition of node resistance curvature. 

Let $G = G(V,E)$ be a graph. Let $A$ be the (unnormalized) adjacency matrix of $G$, let $D$ be the diagonal degree matrix of $G$, and let $L = D - A$ be the (unnormalized) Laplacian matrix of $G$, where each is indexed by $V$; \textit{e.g.}, $D_{x,x}=\deg x$ and $L_{x,y}=-1$ whenever $xy\in E$. We may distinguish graphs using superscripts; \textit{e.g.}, for the graph $G^{(1)}$, vertex degrees are denoted by $\deg^{(1)}\cdot$. 

For vertices $x,y\in V$, let $\gw_{x,y}$ be the {\em effective resistance} across $x$ and $y$, computed as 
\begin{equation}
    \gw_{x,y} = (\bra{x}-\bra{y})L^ + (\ket{x}-\ket{y})=L^+_{x,x}+L^+_{y,y}-2L^+_{x,y}\label{eq:eff res}
\end{equation}
where $L^+$ is the Moore--Penrose pseudoinverse of $L$: 
$$\lpr{\sum\limits_{1\le j\le\#V}\gl_j\ket{\bv_j}\bra{\bv_j}}^+=\sum\limits_{\substack{1\le j\le\#V\\\gl_j\neq0}}\frac{1}{\gl_j}\ket{\bv_j}\bra{\bv_j}.$$
Then, we define the {\em node resistance curvature} as 
$$p_x = 1-\half\sum\limits_{y\sim x}\gw_{x,y}$$
for each node $x$.
We collect the node curvatures into the vector $\bp$, indexed by $V$. 
The vector $\bp$ obeys the rule \cite{DL} that
$$\sum\limits_{x\in V}p_x=1.$$
The (node resistance) curvature of the graph is said to be the value of the least entry in $\bp$. 

We give two examples to illustrate node resistance curvature. 

\begin{example}\label{ex:vtxtransitive}
If $G$ is vertex-transitive, then by symmetry $\bp=\frac{1}{n}\one$~\cite[Appendix C]{DL}. 
More generally, as noted in \cite{DOS}, the graphs with positive constant curvature are the \textit{resistance-regular graphs}~\cite{ZWB}. The family of resistance-regular graphs contains the family of walk-regular graphs, examples of which include the vertex-transitive and distance-regular graphs.
\end{example}

\begin{example}\label{ex:path}
The path graph $P_n$ has node resistance curvature $\half$ at the end nodes and 0 in the interior nodes \cite[\S3.1, Example 2]{DL}. 
\end{example}

In this paper, we study the behavior of the curvature $\bp$ under the operation of Cartesian product of graphs. As a first example, since the Cartesian product of two vertex-transitive graphs $G$ and $H$ is vertex-transitive, it follows as in \exref{ex:vtxtransitive} that $G\square H$ has constant positive curvature. 

While the vertex-transitive case is rendered somewhat degenerate by its inherent symmetry, we still wonder about the general case of Cartesian graph products. It is natural to wonder whether the product of two nonnegatively-curved graphs is always nonnegative. We study here the Cartesian product of paths, giving rise to {\em grids}. The path graph $P_n$ on $n$ vertices has 0 node resistance curvature on its ``interior'' vertices and positive node resistance curvature on its ``boundary'' vertices, where the interior vertices are precisely those with degree 2. We conjecture that a form of this behavior continues to hold for the product of two path graphs, in which the interior becomes (strictly) negative while the boundary remains (strictly) positive. The most general result we conjecture is: 

\begin{conjecture}\label{conj:general-product}
    For $G^{(1)}$ and $G^{(2)}$ any graphs, if $\bp_i^{(1)},\bp_j^{(2)}\le0$ then $\bp_{i\otimes j}<0$. 
\end{conjecture}

This conjecture would partly imply the following: 

\begin{conjecture}\label{conj:path-product}
    For the product of two paths, the curvature $\bp$ is nonpositive on the interior and nonnegative on the boundary.
\end{conjecture}

We prove \conjref{conj:path-product} for graphs $P_m\square P_n$, where $m,n \ge 3$: 
\begin{theorem}\label{thm:grid curv}
    Consider the grid graph $P_m\square P_n$ with $m,n\ge3$. Then, the interior vertices all have negative node resistance curvature, and the boundary vertices all have nonnegative node resistance curvature. If further either $m>3$ or $n>3$ then the node resistance curvature on the boundary is positive with saturatable lower bound $\frac{17}{4830}\approx0.003$. 
\end{theorem}

The paper is organized as follows. We provide some background and notation in \secref{sec:bg}. The proof of \thmref{thm:grid curv} is given in \secref{sec:wide}. We also study the node resistance curvatures of the ladder graphs $P_2 \square P_n$, where the results are somewhat different. For $P_2 \square P_n$, the nodes with positive node resistance curvature are only the ``corner'' nodes (\propref{prop:ladder} in \secref{sec:ladder}). In \secref{sec:gen} we present weak but general bounds and suggest some future directions in this area. 

\section{Background and notation}\label{sec:bg}

\subsection{Cartesian graph products}

Let $G = G^{(1)}\square G^{(2)}$ be the Cartesian product of two simple unweighted graphs $G^{(i)}(V^{(i)},E^{(i)})$, where the vertex set is $V^{(1)} \times V^{(2)}$, whose elements are written as $v^{(1)}\otimes v^{(2)}$, and $v^{(1)}\otimes v^{(2)}$'s neighbors are $\lcr{w^{(1)}\otimes v^{(2)}:w^{(1)}\sim v^{(1)}}\sqcup\lcr{v^{(1)}\otimes w^{(2)}:w^{(2)}\sim v^{(2)}}$. 
(We use the notation $v\otimes w \in V \times W$ rather than $(v, w) \in V \times W$.)

For a finite set $S$, let $\id_S$ be the identity map on the vector space $\R^S$. We also sometimes use bra-ket notation by granting $\R^S$ the Kronecker basis $\lcr{\ket{s}:s\in S}$, where $\norm{\ket{s}}^2=\braket{s}{s}=1$ and $\braket{s}{t}=\gd_{s,t}$ (which extends linearly to the standard inner product). Also $\bra{s}$ is the adjoint of $\ket{s}$, and we form outer products as $\ket{s}\bra{t}$. We also write other norm-1 elements of $\R^S$ as kets, e.g.\ $\ket{\bv}$, to highlight the normalization. 

From the definition of the Cartesian product, we can compute that 
\begin{align*}
    A&=A^{(1)}\otimes\id_{V^{(2)}}+\id_{V^{(1)}}\otimes A^{(2)}\\
    L&=L^{(1)}\otimes\id_{V^{(2)}}+\id_{V^{(1)}}\otimes L^{(2)}.
\end{align*}
where $M \otimes N$ denotes the Kronecker (tensor) product. 

Let a graph's eigenpairs refer to the eigenpairs of its unnormalized Laplacian. Then, if $G^{(i)}$ has eigenpairs $\lcr{\lpr{\gl^{(i)}_j,\ket{\bv^{(i)}_j}}:1\le j\le\#V^{(i)}}$ for $i = 1, 2$, with $\gl^{(i)}_j$ increasing in $j$, then $G$ has eigenpairs
$$\lcr{\lpr{\gl^{(1)}_{j_1}+\gl^{(2)}_{j_2},\ket{\bv^{(1)}_{j_1}}\otimes\ket{\bv^{(2)}_{j_2}}}:1\le j_i\le\#V^{(i)},i\in\{1,2\}}.$$

\subsection{Basic properties of electrical resistance}

Recall the definition of effective resistance \eqref{eq:eff res}. It obeys the {\em series} and {\em parallel} laws. Consider a multigraph $G(V,E)$ with edge weights (resistances) stored as $\ell_e$ for $e\in E$. 
\begin{itemize}
    \item The series law says that if any edge is replaced by a subdivision preserving the total length of the edge, then this does not change any other effective resistance calculations in the graph. 
    \item The parallel law says that if any edge $e$ of length $\ell_e$ is replaced by edges $e_1,\dots,e_n$ of lengths $\ell_{e_1},\dots,\ell_{e_n}$ such that $\ell_e\inv=\ell_{e_1}\inv+\cdots+\ell_{e_n}\inv$, then this does not change any other effective resistance calculations in the graph. 
\end{itemize}

We also have the following principle, intuitive from a physical understanding of circuitry. 
\begin{proposition}[{Rayleigh's monotonicity law, cf.\ \cite{DS}}]
    Consider a subgraph $G\subset H$. For any vertices $i,j\in V(G)$, we have that $\gw^G_{i,j}\ge\gw^H_{i,j}$. 
\end{proposition}
It immediately follows that:
\begin{proposition}[Curvature monotonicity]\label{prop:k2}
    Consider a subgraph $G\subset H$. For any vertex $i\in V(G)$ with $\deg_Gi=\deg_Hi$, we have that $p^G_i\le p^H_i$. 
\end{proposition}

We also use the following classical circuitry result, where $\Z$ implicitly is the graph with $V=\Z$ and edges connecting consecutive integers: 
\begin{proposition}
    The infinite grid $\Z\square\Z$ has 0 node curvature everywhere.
\end{proposition}
This follows immediately from the effective resistance across an edge being exactly $\half$, a folklore result; one proof is given in \cite{Mungan}. 

\section{Products of paths}

We consider Cartesian products of path graphs. Call a graph of the form $P_2\square P_n$ a {\em ladder}, and a graph of the form $P_{n_1}\square\cdots\square P_{n_d}$ with $n_1,\dots,n_d\ge3$ a {\em wide} path product or grid. We fully analyze the resistance curvature of such graphs only after first studying larger two-dimensional grids, and then settle a natural question about wide higher-dimensional grids. 

In a grid $P_{n_1} \square \cdots \square P_{n_d}$, we say a vertex is in the {\em boundary} if its projection to some factor $P_{n_j}$ is an endpoint, and in the {\em interior} otherwise.

\subsection{Wide grids and wide path products}\label{sec:wide}

In this section, we will use the Mathematica commands \texttt{GridBoundaryNodeCurvatures} and \texttt{AllNodeCurvaturesInProduct}, defined in \secref{sec:code}. We prove \thmref{thm:grid curv}, which we recall shows, among other things, the boundary vertices of grid graph $P_m \square P_n$ have nonnegative node curvature when $m, n\ge 3$. 

\begin{proof}[{Proof of \thmref{thm:grid curv}}]
    We prove the result about interior vertices first. Take any interior vertex $i$, and consider the inputs $G=P_m\square P_n$ and $H=\Z\square\Z$ to \propref{prop:k2}, where we embed $G$ into $H$ in any way; say, identify $V(G) = \lcr{i\otimes j:1\le i\le m,1\le j\le n}$. The claim about interior vertex curvature follows. 

    Write now $H=P_m\square P_n$. For the boundary vertices, again model $V(H)$ as $V(G)$ was just before. Up to symmetry (reflecting and rotation, possibly swapping the role of $m$ and $n$), assume that a given boundary vertex $(i,j)$ has $i=1$ and $j\le\frac{n}{2}$. If $j=1$ then let put $G=P_3\square P_3$ modeled as $V(G)=\lcr{i\otimes j:1\le i,j\le 3}\subseteq V(H)$. (Think of placing the $3\times3$ box inside $H$ along the edge, sliding it until it contains $(i,j)$; see \figref{fig:inset}.) 
    \begin{figure}
        \centering
        \includegraphics[width=0.5\textwidth]{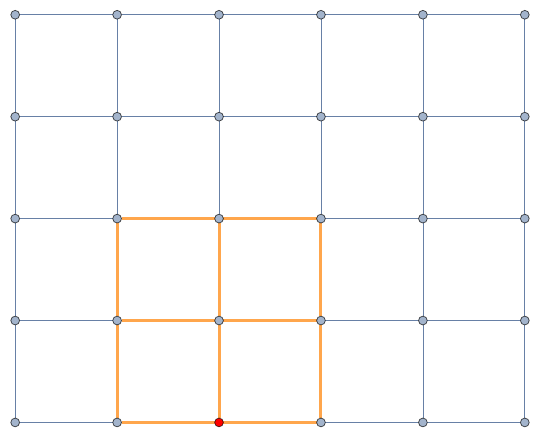}
        \caption{The $3\times3$ box (orange) is slid until it contains the target vertex (red).}
        \label{fig:inset}
    \end{figure}Clearly we may again apply \propref{prop:k2}, with inputs $i\otimes j$, $G$, and $H$, to find
    $$p^G_{i\otimes j}\le p^H_{i\otimes j}.$$
    However $G$ is a sufficiently small graph that we may compute with it directly, and we find that each node curvature is at least 0 (see \figref{fig:33}), i.e.\ $p^G_{i\otimes j}\ge0$.
    \begin{figure}
        \centering
        \includegraphics[width=0.5\textwidth]{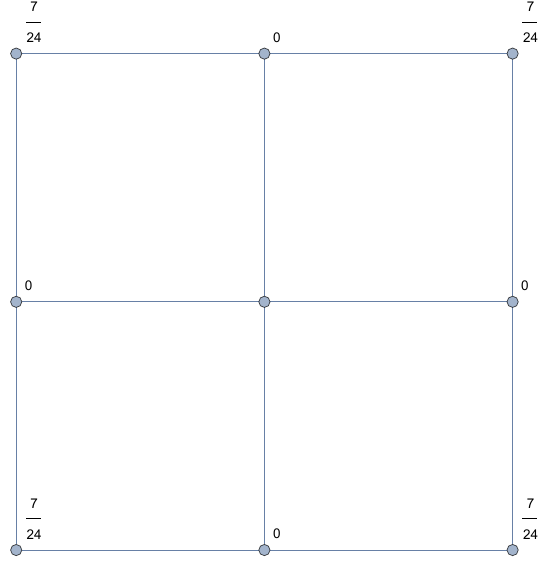}
        \caption{Boundary curvatures for the $3\times3$ grid. Produced with the Mathematica command \texttt{GridGraph[\string{3, 3\string}, VertexLabels -> GridBoundaryNodeCurvatures[3, 3]]}.}
        \label{fig:33}
    \end{figure}
    
    Similarly, if $H$ is large enough to allow a $3\times4$ path to be embedded, then doing the same computation with $P_3\square P_4$ gives \figref{fig:34} and in this setting that $p^G_{i\otimes j}\ge\frac{17}{4830}$. 
    \begin{figure}
        \centering
        \includegraphics[width=0.5\textwidth]{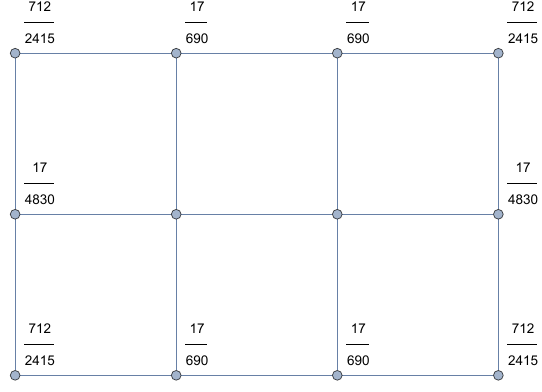}
        \caption{Boundary curvatures for the $3\times4$ grid. Produced with the Mathematica command \texttt{GridGraph[\string{3, 4\string}, VertexLabels -> GridBoundaryNodeCurvatures[3, 4]]}.}
        \label{fig:34}
    \end{figure}
\end{proof}

It is natural to wonder next whether this behavior persists into higher-dimensional grids. It actually happens that this is unique to dimension 2, as witnessed by the smallest possible instance. In three dimensions, we let an {\em interior} vertex be one with degree exactly 6. (In general, it only makes sense to call the interior vertex in a $d$-fold product of paths to be one of degree exactly $2d$.)

\begin{proposition}
    The graph $P_3\square P_3\square P_3$ has boundary vertices with negative node curvature. 
\end{proposition}

This can be seen from a Mathematica command\footnote{\texttt{AllNodeCurvaturesInProduct[Normal[KirchhoffMatrix[PathGraph[Range[3]]]], 3]}} where since there is a unique interior vertex (the one corresponding to the middle of each of the constituent paths) we are done once we recognize more than one negative number in the output. 

\subsection{Ladder graphs}\label{sec:ladder}

\begin{figure}
    \centering
    \begin{subfigure}[b]{0.3\textwidth}
        \centering
        \includegraphics[width=\textwidth]{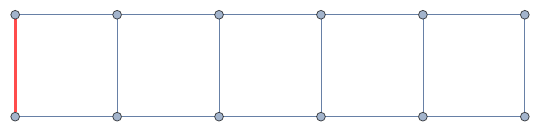}
        \caption{A rung at the top/bottom.}
        \label{fig:end rung}
    \end{subfigure}
    \begin{subfigure}[b]{0.3\textwidth}
        \centering
        \includegraphics[width=\textwidth]{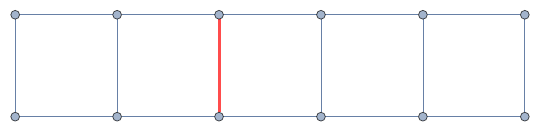}
        \caption{A rung in the middle.}
        \label{fig:mid rung}
    \end{subfigure}
    \begin{subfigure}[b]{0.3\textwidth}
        \centering
        \includegraphics[width=\textwidth]{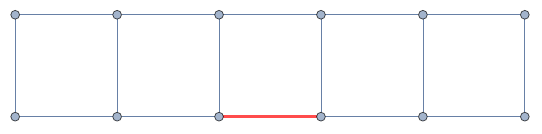}
        \caption{A rail in the middle.}
        \label{fig:mid rail}
    \end{subfigure}
    \caption{Examples of rungs and rails in ladders, highlighted in red.}
    \label{fig:rung rail pics}
\end{figure}

We will study node curvature of the ladder graph $G^{(n)}=P_2\square P_n$. Consider the vertex set to be $\{1,2\}\otimes\{1,\dots,n\}$. Call the edge $e^{(n)}_k$ connecting $1\otimes k$ and $2\otimes k$ the {\em $k$th rung}, and the edge $e^{(n)}_{i,k}$ connecting $i\otimes k$ to $i\otimes(k+1)$, for $i\in\{1,2\}$ and $1\le k<n$, the {\em ($i,k$)th rail}. See Figures \ref{fig:mid rung} and \ref{fig:mid rail} for examples. We prove the following about node curvature in ladders: 

\begin{proposition}\label{prop:ladder}
    For $1<k<n$, $\bp^{(n)}_{i\otimes k}<0$, 
    and both $\bp^{(n)}_{i\otimes1}$ and $\bp^{(n)}_{i\otimes n}$ increase monotonically and converge to $2 - \sqrt{3}$ as $n \to \infty$.
\end{proposition}

We do this by computing the effective resistance across rungs and rails (Lemmas \ref{lem:end rung res}, \ref{lem:gen rung res}, and \ref{lem:gen rail res}). 
\begin{lemma}\label{lem:end rung res}
    The effective resistance $\ga_n=\gw^{(n)}_{e^{(n)}_1}$ has $\ga_n\downarrow\sqrt{3}-1$ and obeys the recurrence
    \begin{equation}
        \ga_{n+1}=\frac{\ga_n+2}{\ga_n+3}.\label{eq:rung rec}
    \end{equation}
\end{lemma}
\begin{proof}
    (See \figref{fig:end rung}.) The resistance across this rung in $G^{(n+1)}$ is the same as the resistance across two nodes with 1 and $\ga_n+2$ resistors, so by the parallel law, 
    $$\ga_{n+1}=\frac{1}{1+\frac{1}{\ga_n+2}}$$
    so we recover \eqref{eq:rung rec}. Then, we observe that $\ga_1=1$ and the map $f:x\longmapsto\frac{x+2}{x+3}$ decreases for $x>\sqrt{3}-1$, so that its iterates $\lpr{f^{(n)}(1)}_{n\in\N}$ form a decreasing bounded sequence; thus it has a limit, and by continuity $\sqrt{3}-1$ can be checked to be the only possible limit. 
\end{proof}
\begin{lemma}\label{lem:gen rung res}
    The effective resistance across the $k$th rung equals
    $$\frac{1}{1+\frac{1}{\ga_{k-1}+2}+\frac{1}{\ga_{n-k}+2}}.$$
\end{lemma}
\begin{proof}
    (See \figref{fig:mid rung}.) The resistance across this rung in $G^{(n)}$ is the same as the resistance across two nodes with 1, $\ga_{k-1}+2$, and $\ga_{n-k}+2$ resistors, so by the parallel law the result follows. 
\end{proof}
\begin{lemma}\label{lem:gen rail res}
    The effective resistance across the ($i,1$)th rail equals $\ga_n$. For $1<k<n$, the effective resistance across the ($i,k$)th rail equals
    $$\frac{1}{1+\frac{1}{\ga_k+\ga_{n-k}+1}}.$$
\end{lemma}
\begin{proof}
    (See \figref{fig:mid rail}.) The resistance across this rail in $G^{(n)}$ is the same as the resistance across two nodes with 1 and $\ga_k+\ga_{n-k}+1$ resistors, so by the parallel law the result follows. 
\end{proof}
\begin{figure}
    \centering
    \begin{subfigure}[b]{0.45\textwidth}
        \centering
        \includegraphics[width=\textwidth]{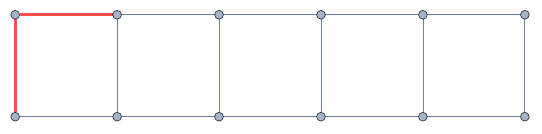}
        \caption{The edges incident to a ``corner'' vertex.}
        \label{fig:corner incidents}
    \end{subfigure}
    \begin{subfigure}[b]{0.45\textwidth}
        \centering
        \includegraphics[width=\textwidth]{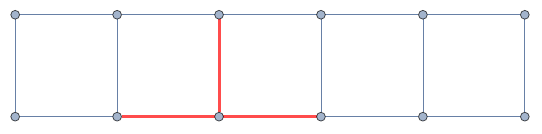}
        \caption{The edges incident to a ``generic'' vertex.}
        \label{fig:gen incidents}
    \end{subfigure}
    \caption{Possibilities for edges incident to vertices in ladders, highlighted in red.}
    \label{fig:ladder incidents}
\end{figure}
\begin{proof}[{Proof of \propref{prop:ladder}}]
    (See \figref{fig:corner incidents}.) When $1<k<n$, by Lemmas \ref{lem:gen rung res} and \ref{lem:gen rail res},
    \begin{align*}
        \bp^{(n)}_{i\otimes k}&=1-\half\lpr{\gw^{(n)}_{e^{(n)}_k}+\gw^{(n)}_{e^{(n)}_{i\otimes(k-1)}}+\gw^{(n)}_{e^{(n)}_{i\otimes k}}}\\
            &=-\half\frac{(\ga_{k-1}+1)(\ga_{n-k}+1)-3}{(\ga_{k-1}+3)(\ga_{n-k}+3)-1}
    \end{align*}
    which we recognize as negative since $\ga_{k-1},\ga_{n-k}>\sqrt{3}-1$. 

    (See \figref{fig:gen incidents}.) When $k=1$, 
    $$\bp^{(n)}_{i\otimes1}=1-\ga_n$$
    and so we apply \lemref{lem:end rung res}. The same holds for $\bp^{(n)}_{i\otimes n}$ by symmetry. 
\end{proof}

\section{General bounds and future questions}\label{sec:gen}

It remains unresolved how to understand the curvature in more general graph products. A first step towards this is to be able to bound, if not compute exactly, the effective resistance across individual edges in a graph product. As before, consider $G=G^{(1)}\square G^{(2)}$ and keep all other notation. We consider the edge $e$ connecting $v^{(1)}\otimes v^{(2)}$ and $w^{(1)}\otimes v^{(2)}$, where $e^{(1)}$ connects $v^{(1)}$ and $w^{(1)}$ in $E^{(1)}$. 

To obtain an upper bound on $\gw_e$ in terms of $\gw=\gw^{(1)}_{e^{(1)}}$, we suppose that there is a $r^{(2)}$-depth $d^{(2)}$-regular tree rooted at $v^{(2)}$ in $G^{(2)}$. Taking the product with the graph which is a single edge of resistance $\gw$, we can see that the effective resistance across the edge connecting the two roots is $f(r^{(2)})$, where $f(0)=\gw$ and $f(k)\inv=\gw\inv+(d-1)(2+f(k-1))\inv$. (For instance, the case $r^{(2)}=1$ and $d^{(2)}=5$ is depicted in \figref{fig:star}.) We then apply curvature monotonicity, and in the case $r^{(2)}=1$ recover 
\begin{equation}
    \gw_e\le\gw\frac{1+\frac{2}{\gw}}{d+1+\frac{2}{\gw}}.\label{eq:ub}
\end{equation}
This bound is not so interesting when $\frac{1}{\gw}\gg d$ but is more powerful for larger $\gw$. 
\begin{figure}
    \centering
    \includegraphics[width=0.5\textwidth]{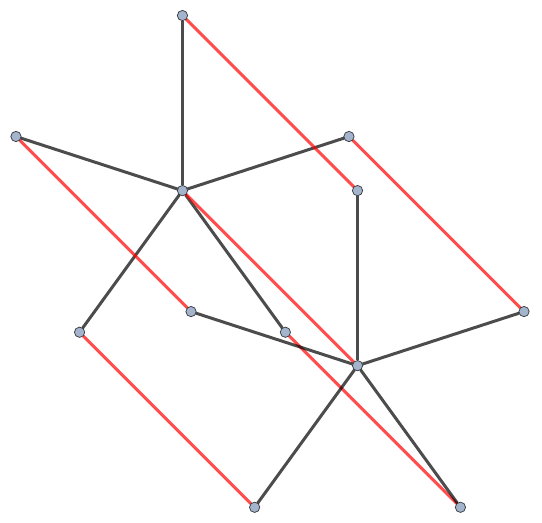}
    \caption{A neighborhood of a vertex, times an edge, in the degree-5 case. The red edges are all of resistance $\gw$ and the black edges are all of weight 1.}
    \label{fig:star}
\end{figure}

To obtain a lower bound we work directly with 
$$L^+={\sum_{\substack{j^{(1)},j^{(2)}\\\text{$\gl_{j^{(1)}}^{(1)}$, $\gl_{j^{(2)}}^{(2)}$ not both 0}}}}\frac{1}{\gl_{j^{(1)}}^{(1)}+\gl_{j^{(2)}}^{(2)}}\ket{\bv^{(1)}_{j^{(1)}}}\bra{\bv^{(1)}_{j^{(1)}}}\otimes\ket{\bv^{(2)}_{j^{(2)}}}\bra{\bv^{(2)}_{j^{(2)}}}$$
and do casework on whether either of $\gl_{j^{(1)}}^{(1)}$ or $\gl_{j^{(2)}}^{(2)}$ are 0. From this, if $n^{(2)}=\#V^{(2)}$, then one may show from this analysis that 
\begin{equation}
    \gw_e\ge\lpr{\frac{1}{n^{(2)}}+\lpr{1-\frac{1}{n^{(2)}}}\frac{\gl^{(1)}_2}{\gl^{(1)}_2+\gl^{(2)}_{n^{(2)}}}}\gw.\label{eq:lb}
\end{equation}

Unfortunately, both of these bounds \eqref{eq:ub} and \eqref{eq:lb} appear to be too weak to prove any results of interest. We hope to be able to better understand effective resistances in these settings. Another direction of interest could be to graphs with non-uniform resistances; the derivation of \eqref{eq:lb} already allows for different resistances, as well as different resistances in $G^{(1)}$ for \eqref{eq:ub}, but there is much to be explored. 

\section*{Acknowledgements}

This work started at the 2023 American Mathematical Society Mathematical Research Communities on Ricci Curvatures of Graphs and Applications to Data Science, which was supported by the National Science Foundation under Grant Number DMS 1916439. 
We thank Fan Chung, Mark Kempton, Wuchen Li, Linyuan Lu, and Zhiyu Wang for organizing this workshop.
Our investigation was greatly assisted by the online graph curvature calculator \cite{CKLLS-calculator}.
WL was also partially supported by NSF RTG Grant DMS 2038080. ZS was additionally supported by NSF grant DGE 2146752.

\bibliographystyle{alpha}
\bibliography{resistance}

\newcommand{\etalchar}[1]{$^{#1}$}
\begin{thebibliography}{CKL{\etalchar{+}}22}

\bibitem[Chu97]{Chung}
Fan R.~K. Chung.
\newblock {\em Spectral graph theory}, volume~92 of {\em CBMS Regional
  Conference Series in Mathematics}.
\newblock Conference Board of the Mathematical Sciences, Washington, DC; by the
  American Mathematical Society, Providence, RI, 1997.

\bibitem[CKL{\etalchar{+}}22]{CKLLS-calculator}
David Cushing, Riikka Kangaslampi, Valtteri Lipi{\"a}inen, Shiping Liu, and
  George~W. Stagg.
\newblock The graph curvature calculator and the curvatures of cubic graphs.
\newblock {\em Exp. Math.}, 31(2):583--595, 2022.

\bibitem[DL22]{DL}
Karel Devriendt and Renaud Lambiotte.
\newblock Discrete curvature on graphs from the effective resistance.
\newblock {\em Journal of Physics: Complexity}, 3(2):025008, 2022.

\bibitem[DOS24]{DOS}
Karel Devriendt, Andrea Ottolini, and Stefan Steinerberger.
\newblock Graph curvature via resistance distance.
\newblock {\em Discrete Applied Mathematics}, 348:68--78, 2024.

\bibitem[DS84]{DS}
Peter~G. Doyle and J.~Laurie Snell.
\newblock {\em Random walks and electric networks}, volume~22 of {\em Carus
  Math. Monogr.}
\newblock Mathematical Association of America, Washington, DC, 1984.

\bibitem[For03]{forman}
Robin Forman.
\newblock Bochner's method for cell complexes and combinatorial {Ricci}
  curvature.
\newblock {\em Discrete Comput. Geom.}, 29(3):323--374, 2003.

\bibitem[LLY11]{LLY}
Yong Lin, Linyuan Lu, and Shing-Tung Yau.
\newblock Ricci curvature of graphs.
\newblock {\em T{\^o}hoku Math. J. (2)}, 63(4):605--627, 2011.

\bibitem[Mun99]{Mungan}
C.E. Mungan.
\newblock {Infinite Square Lattice of Resistors}.
\newblock Note available at
  \url{https://www.usna.edu/Users/physics/mungan/_files/documents/Scholarship/ResistorLattice.pdf},
  1999.

\bibitem[Oll09]{ollivier}
Yann Ollivier.
\newblock Ricci curvature of {Markov} chains on metric spaces.
\newblock {\em J. Funct. Anal.}, 256(3):810--864, 2009.

\bibitem[ZWB16]{ZWB}
Jiang Zhou, Zhongyu Wang, and Changjiang Bu.
\newblock On the resistance matrix of a graph.
\newblock {\em Electron. J. Combin.}, 23(1):Paper 1.41, 10, 2016.

\end{thebibliography}
\appendix
\section{Code}\label{sec:code}

Here is the Mathematica code used in \secref{sec:wide}. 
\begin{lstlisting}[language=Mathematica]
tens = KroneckerProduct;
EffectiveResistances[G_] :=
 Module[{L = KirchhoffMatrix[G](* unnormalized Laplacian *), 
   n = Length[VertexList[G]], PI},
  PI = Inverse[L + ConstantArray[1/n, {n, n}]] - 
    ConstantArray[1/n, {n, n}](* inverse of shifted Laplacian *);
  Table[PI[[i, i]] + PI[[j, j]] - 2 PI[[i, j]], {i, 1, n}, {j, 1, 
    n}] (* effective resistances *)
  ]
GridBoundaryNodeCurvatures[m_, n_] := Module[
  {G = GridGraph[{m, n}], 
   B = Select[
     Flatten[Table[{a, b}, {a, 0, n - 1}, {b, 1, m}], 
      1], (#[[1]] == 0 || #[[1]] == n - 1 || #[[2]] == 1 || #[[2]] == 
         m) &], E, nbhd, convert},
  E = EffectiveResistances[G];
  nbhd[v_] := 
   With[{a = v[[1]], b = v[[2]]}, 
    If[a > 0, {{a - 1, b}}, {}] \[Union] 
     If[a < n - 1, {{a + 1, b}}, {}] \[Union] 
     If[b > 1, {{a, b - 1}}, {}] \[Union] If[b < m, {{a, b + 1}}, {}]];
  convert[v_] := With[{a = v[[1]], b = v[[2]]}, a m + b];
  Table[
   convert[v] ->
    1 - 1/2 Total[Table[E[[convert[v], convert[w]]], {w, nbhd[v]}]]
   , {v, B}]
  ]
AllNodeCurvaturesInProduct[L_, d_](* 
  L is the Laplacian of a graph, whose d-
  fold Cartesian product is to be taken *):=
  Module[{Lap, it, PI, n = Length[L]^d, O},
   it[A_, k_] := 
    If[k <= 1, A, 
     With[{a = Length[A]}, 
      tens[it[A, k - 1], IdentityMatrix[a]] + 
       tens[IdentityMatrix[a^(k - 1)], A]]];
   Lap = it[L, d];
   PI = Inverse[Lap + ConstantArray[1/n, {n, n}]] - 
     ConstantArray[1/n, {n, n}](* inverse of shifted Laplacian *);
   O = 
    Table[PI[[i, i]] + PI[[j, j]] - 2 PI[[i, j]], {i, 1, n}, {j, 1, 
      n}] (* effective resistances *);
   Table[1 + 1/2 Dot[O[[i]], Lap[[i]]], {i, 1, n}]
   ];
\end{lstlisting}

\end{document}